\documentclass{amsart}
\usepackage{amssymb}

\newcommand{\Q}{\mathbb{Q}}
\newcommand{\Z}{\mathbb{Z}}
\newcommand{\F}{\mathbb{F}}
\newcommand{\cE}{\mathcal{E}}
\newcommand{\cP}{\mathcal{P}}
\newcommand{\R}{\mathbb{R}}
\newcommand{\fp}{\mathfrak{p}}
\newcommand{\fq}{\mathfrak{q}}
\newcommand{\cN}{\mathcal{N}}
\newcommand{\bs}{\mathbf{s}}
\newcommand{\OO}{\mathcal{O}}

\DeclareMathOperator{\Gal}{Gal}
\DeclareMathOperator{\Aut}{Aut}
\DeclareMathOperator{\GL}{GL}
\DeclareMathOperator{\Norm}{Norm}
\DeclareMathOperator{\Res}{Res}
\DeclareMathOperator{\lcm}{lcm}

\newtheorem{thm}{Theorem}
\newtheorem{lem}{Lemma}[section]
\newtheorem{prop}[lem]{Proposition}
\newtheorem{cor}[lem]{Corollary}

\numberwithin{equation}{section}

\begin{document}

\title[Irreducibility of mod $p$ Representations]{Criteria for Irreducibility of mod $p$ Representations of Frey Curves}
 
\author{\sc Nuno Freitas}
\address{Mathematisches Institut\\
Universit\"{a}t Bayreuth\\
95440 Bayreuth, Germany}
\email{nunobfreitas@gmail.com}

\author{\sc Samir Siksek}
\address{Mathematics Institute\\
University of Warwick\\
CV4 7AL\\
United Kingdom}
\email{samir.siksek@gmail.com}
\thanks{
The first-named author is supported
through a grant within the framework of the DFG Priority Programme 1489
{\em Algorithmic and Experimental Methods in Algebra, Geometry and Number Theory}.
The second-named
author is supported by an EPSRC Leadership Fellowship EP/G007268/1.
}



\maketitle

\begin{abstract}
Soit $K$ un corps galoisien totalement r\'eel, et soit $\cE$ un
ensemble de courbes elliptiques sur $K$. Nous donnons des
conditions suffisantes pour l'existence d'un ensemble calculable
de premiers rationnels $\cP$ tels que, pour $p \notin \cP$ et 
$E \in \cE$, la repr\'esentation $\Gal(\overline{K}/K) \rightarrow \Aut(E[p])$
soit irr\'eductible. Nos conditions sont en g\'en\'eral satisfaites par les courbes
de Frey associ\'ees \`a des solutions d'\'equations diophantiennes; dans ce
contexte, l'irr\'eductibilit\'e de la mod $p$ repr\'esentation est une hypoth\`ese
requise pour l'application des th\'eor\`emes d'abaissement du niveau. Comme
illustration de notre approche, nous avons am\'elior\'e le r\'esultat de \cite{DF2}
pour les \'equations de Fermat de signature $(13, 13, p)$. 
\end{abstract}

\begin{abstract}
 Let $K$ be a totally real Galois number field and let $\cE$ be a set of
elliptic curves over $K$. We give sufficient conditions
for the existence of a finite computable set of rational primes $\cP$
such that for $p \notin \cP$ and $E \in \cE$, the representation
$\Gal(\overline{K}/K) \rightarrow \Aut(E[p])$ is irreducible.
Our sufficient conditions are often satisfied for Frey elliptic curves
associated to solutions of Diophantine equations; in that context,
the irreducibility of the mod $p$ representation is 
a hypothesis needed for applying level-lowering theorems.
We illustrate our approach by improving on a result of \cite{DF2} 
for Fermat-type equations of signature $(13,13,p)$.
\end{abstract}

\bigskip

\section{Introduction}
The \lq modular approach\rq\ is a popular method for attacking
Diophantine equations using Galois representations of 
elliptic curves; see \cite{BCDY}, \cite{Siksek} for
recent surveys. 
The method relies on three important
and difficult theorems.
\begin{enumerate}
\item[(i)] Wiles et al.:
elliptic curves over $\Q$ are modular \cite{modularity}, \cite{Wiles},
\cite{TW}.
\item[(ii)] Mazur: if $E/\Q$ is an elliptic curve
and $p>167$ is a prime, then the Galois representation
on the $p$-torsion of $E$ is irreducible \cite{Mazur} (and variants
of this result).
\item[(iii)] Ribet's level-lowering theorem \cite{Ribet}.
\end{enumerate}
The strategy of the method is to 
associate to a putative solution of certain 
Diophantine equations a Frey elliptic curve, and apply 
Ribet's level-lowering theorem 
to deduce a relationship
between the putative solution and a modular form
of relatively small level. Modularity (i) and irreducibility (ii)
are necessary hypotheses that need to be verified in order
to apply level-lowering (iii). 

Attention is now shifting towards Diophantine
equations where the Frey elliptic curves are 
defined over totally real fields (for example
\cite{BDMS}, \cite{DF2}, \cite{F}, \cite{FS}).
One now finds in the literature some of the necessary
modularity (e.g.\ \cite{FHS}) and level-lowering
theorems (e.g.\ \cite{Fuj}, \cite{Jarv} and \cite{Raj})
for the totally real setting. Unfortunately, there is 
as of yet no analogue of Mazur's Theorem over any
number field $K \ne \Q$, which does present an
obstacle for applying the modular approach over
totally real fields. 

\bigskip

Let $K$ be a number field,
and write $G_K=\Gal(\overline{K}/K)$.
Let $E$ an elliptic curve over $K$.
Let $p$ be a rational prime, and write $\overline{\rho}_{E,p}$
for the associated representation of $G_K$ on the $p$-torsion of $E$:
\begin{equation}\label{eqn:rho}
\overline{\rho}_{E,p}\; :\; G_K \rightarrow \Aut(E[p]) \cong \GL_2(\F_p).
\end{equation}
Mazur's Theorem asserts that if $K=\Q$ and $p>167$
then $\overline{\rho}_{E,p}$ is irreducible. 
For a general number field $K$, it is expected that there
is some $B_K$, such that for all elliptic curves $E/K$
without complex multiplication, and all $p>B_K$,
the mod $p$ representation $\overline{\rho}_{E,p}$
is irreducible. Several papers, including those
by Momose \cite{Momose}, Kraus \cite{KrausQuad}, \cite{KrausIrred}
and David \cite{DavidI}, establish a bound $B_K$
depending on the field $K$, under some restrictive assumptions 
on $E$, such as semistability. 
The Frey elliptic
curves one deals with in the modular approach are
close to being semistable \cite[Section 15.2.4]{Siksek}.
The purpose of this note is to prove the following theorem,
which should usually be enough to supply the 
desired irreducibility
statement in that setting.
\begin{thm}\label{thm:main}
Let $K$ be a totally real Galois number field of degree $d$, 
with ring of integers $\OO_K$ and Galois group $G=\Gal(K/\Q)$.
Let $S=\{0,12\}^G$, which we think of as the set of sequences of values $0$, $12$
indexed by $\tau \in G$. 
For $\bs=(s_\tau) \in S$ and $\alpha \in K$, define the \textbf{twisted norm associated
to $\bs$} by
\[
\cN_\bs(\alpha)= \prod_{\tau \in G} \tau(\alpha)^{s_\tau}.
\]
Let $\epsilon_1,\dots,\epsilon_{d-1}$
be a basis for the unit group of $K$,
and define
\begin{equation}\label{eqn:As}
A_\bs:=\Norm  \left( \gcd ( ( \cN_\bs(\epsilon_1)-1) \OO_K,\ldots, (\cN_\bs(\epsilon_{d-1})-1  ) \OO_K) \right).
\end{equation}
Let $B$ be the least common multiple of the $A_\bs$ taken over all $\bs \ne (0)_{\tau \in G}$,
$(12)_{\tau \in G}$. 
Let $p \nmid B$ be a rational prime, unramified in $K$, such that $p \geq 17$ or $p = 11$.
Let $E/K$ be an elliptic curve, and $\fq \nmid p$ be a prime of good reduction for $E$.
Let 
\[
P_\fq(X)=X^2-a_\fq(E) X + \Norm(\fq)
\]
be the characteristic polynomial
of Frobenius for $E$ at $\fq$. Let $r \ge 1$ be an integer such that
$\fq^r$ is principal.
If $E$ is semistable at all $\fp \mid p$ 
and $\overline{\rho}_{E,p}$ is reducible then
\begin{equation}\label{eqn:res}
p \mid \Res(\, P_\fq(X) \, , \, X^{12 r}-1\, )
\end{equation}
where $\Res$ denotes resultant.
\end{thm}
We will see in due course that $B$ above is non-zero. It is easy show
that the resultant in \eqref{eqn:res} is also non-zero. The theorem
therefore does give a bound on $p$ so that $\overline{\rho}_{E,p}$
is reducible.

\bigskip

The main application we have in mind is to Frey elliptic
curves associated to solutions of Fermat-style
equations. In such a setting, one usually knows
that the elliptic curve in question has semistable
reduction outside a given set of primes, and one
often knows some primes of potentially good
reduction. We illustrate this, by giving an
improvement to a recent theorem of 
Dieulefait and Freitas \cite{DF2}
on the equation $x^{13}+y^{13}=C z^p$.
In a forthcoming paper \cite{BDMS},
the authors apply our Theorem~\ref{thm:main}
together with modularity and level-lowering theorems
to completely solve the equation $x^{2n} \pm 6 x^n +1 =y^2$
in integers $x$, $y$, $n$ with $n \ge 2$, after associating
this to a Frey elliptic curve over $\Q(\sqrt{2})$.
In another paper \cite{F}, the first-named
author uses our Theorem~\ref{thm:main} as part of
an investigation that associates solutions 
of equations $x^r+y^r=C z^p$ with ($r$, $p$ prime)
with Frey elliptic curves over real subfields
of $\Q(\zeta_r)$.

\bigskip

The following is closely related to a result of David \cite[Theorem 2]{DavidI},
but formulated in a way that is more suitable for attacking specific examples.
\begin{thm} \label{thm:2}
Let $K$ be a totally real Galois field of degree $d$. Let $B$ be as in the statement of Theorem~\ref{thm:main}.
Let $p \nmid B$ be a rational prime, unramified in $K$, such that $p \geq 17$ or $p = 11$.
If $E$ is an elliptic curve over $K$ which is semistable at all $\fp \mid p$
and $\overline{\rho}_{E,p}$ is reducible then
$p < (1 + 3^{6d h})^2$,
where $h$ is the class number of $K$.	
\end{thm}

\section{Preliminaries}
We shall henceforth fix the following notation and assumptions.

\medskip
\begin{tabular}{lll}
$K$ & $\qquad$ & {a Galois number field,}\\
$d_K$ & & {the degree of $K/\Q$,}\\
$G_K$ & & $\Gal(\overline{K}/K)$,\\
$\fq$ & & {a finite prime of $K$},\\
$I_\fq$ & & {the inertia subgroup of $G_K$ corresponding to $\fq$},\\
$G$ & & $\Gal(K/\Q)$,\\
$p$ & & {a rational prime unramified in $K$ satisfying $p \ge 17$},\\
& & {or $p=11$},\\
$\chi_p$ & & {the mod $p$ cyclotomic character $G_K \rightarrow \F_p^*$},\\
$E$ &  & {an elliptic curve semistable at all places $\fp$ of $K$ above $p$}, \\
$\overline{\rho}_{E,p}$ & & {the mod $p$ representation associated
to $E$ as in \eqref{eqn:rho}}.
\end{tabular}

\bigskip

Suppose $\overline{\rho}_{E,p}$ is reducible. 
With an appropriate choice of basis for $E[p]$
we can write 
\begin{equation}\label{eqn:matrix}
\overline{\rho}_{E,p} \sim 
\begin{pmatrix}
\lambda & * \\
0 & \lambda^\prime
\end{pmatrix} \, ,
\end{equation}
where $\lambda$, $\lambda^\prime \, : \, G_K \rightarrow \F_p^*$ are characters. 
Thus $\lambda \lambda^\prime= \det(\overline{\rho}_{E,p})=\chi_p$. 
The character $\lambda$ is known as the \textbf{isogeny character} of $E[p]$.

As in the aforementioned works of Momose, Kraus and David,
our approach relies on controlling the ramification of the characters
$\lambda$, $\lambda^\prime$ at places above $p$.

\begin{prop} \textup{(David \cite[Propositions 1.2, 1.3]{DavidI})}\label{prop:exponents}
Suppose $\overline{\rho}_{E,p}$ is reducible and let $\lambda$, $\lambda^\prime$
be as above.
Let $\fp \mid p$ be a prime of $K$. 
Then
\[
\lambda^{12} \vert_{I_\fp}= \left(\chi_\fp \vert_{I_\fp} \right)^{s_\fp}
\]
where $s_\fp \in \{0,12\}$. 
\end{prop}
\begin{proof}
Indeed, by \cite[Propositions 1.2, 1.3]{DavidI},
\begin{enumerate}
\item[(i)] if $\fp$ is a prime of potentially multiplicative reduction 
or potentially good ordinary reduction for $E$
then
$s_\fp=0$ or $s_\fp=12$;
\item[(ii)] if $\fp$ is a prime of potentially good supersingular reduction for $E$
then
$s_\fp=4$, $6$, $8$.
\end{enumerate}
However, we have assumed that $E$ is semistable at $\fp \mid p$ and that $p$
is unramified in $K$. By Serre \cite[Proposition 12]{Serre}, if $E$ has 
good supersingular reduction at $\fp$, then the image of $\overline{\rho}_{E,p}$
contains a non-split Cartan subgroup of $\GL_2(\F_p)$
and is therefore irreducible, contradicting the assumption that $\overline{\rho}_{E,p}$
is reducible. Hence $E$ has multiplicative or good ordinary reduction at $\fp$.
\end{proof}
\noindent \textbf{Remark.}
The order of $\chi_\fp \vert_{I_\fp}$ is $p-1$.
Hence the value of $s_\fp$ in the above proposition is well-defined
modulo $p-1$. Of course, since $0 \le s_\fp \le 12$, it follows for $p \ge 17$
that $s_\fp$ is unique.

\bigskip

As $K$ is Galois, $G$ acts transitively of $\fp \mid p$. Fix
$\fp_0 \mid p$. For each $\tau \in G$ write $s_\tau$ 
for the number $s_\fp$ associated to the ideal $\fp:=\tau^{-1} (\fp_0)$
by the previous proposition. We shall refer to $\bs:=(s_\tau)_{\tau \in G}$
as the \textbf{isogeny signature} of $E$ at $p$. The set $S:=\{0,12\}^G$
shall denote the set of all possible sequences of values $0$, $12$
indexed by elements of $G$.
For an element $\alpha \in K$, we define the \textbf{twisted norm associated
to $\bs \in S$} by
\[
\cN_\bs(\alpha)= \prod_{\tau \in G} \tau(\alpha)^{s_\tau}.
\]
\begin{prop} \textup{(David \cite[Proposition 2.6]{DavidI})}\label{prop:lam12}
Suppose $\overline{\rho}_{E,p}$ is reducible with isogeny character
$\lambda$, having isogeny signature $\bs \in S$.
Let $\alpha \in K$ be non-zero. Suppose $\upsilon_\fp(\alpha)=0$
for all $\fp \mid p$. Then
\[
\cN_\bs(\alpha)  
\equiv
\prod \left(\lambda^{12}(\sigma_\fq) \right)^{\upsilon_\fq(\alpha)}
\pmod{\fp_0},
\]
where the product is taken over all prime $\fq$ in the support of $\alpha$.
\end{prop}

\section{A bound in terms of a prime of potentially good reduction}

Let $\fq$ be a prime of potentially good reduction for $E$. 
Denote by $P_\fq(X)$ the characteristic polynomial of 
Frobenius for $E$ at $\fq$.

\begin{lem}\label{lem:good}
Let $\fq$ be a prime of potentially good reduction for $E$,
and suppose $\fq \nmid p$.
Let $r \ge 1$ be such that $\fq^r$ is principal, and write $\alpha \OO_K=\fq^r$.
Let $\bs=(s_\tau)_{\tau \in G}$ be the isogeny
signature of $E$ at $p$. Then
\[
\fp_0 \mid \Res(\,  P_\fq(X) \, ,\, X^{12r}-\cN_\bs(\alpha) \, ),
\]
where $\Res$ denotes the resultant.
\end{lem}
\begin{proof}
From~\eqref{eqn:matrix},
it is clear that 
\[
P_\fq(X) \equiv (X - \lambda(\sigma_\fq)) (X-\lambda^\prime(\sigma_\fq)) \pmod{p}.
\]
Moreover, from Proposition~\ref{prop:lam12}, $\lambda(\sigma_\fq)$
is a root modulo $\fp_0$ of the polynomial
$X^{12 r} - \cN_\bs(\alpha)$. As $\fp_0 \mid p$, the lemma follows.
\end{proof}

We note the following surprising consequence.
\begin{cor}\label{cor:unitbound}
Let $\epsilon$ be a unit 
of $\OO_K$. If the isogeny signature of $E$ at $p$ is $\bs$ then
$\cN_\bs(\epsilon) \equiv 1 \pmod{\fp_0}$.
\end{cor}
\begin{proof}
Let $\fq \nmid p$ be any prime of good reduction of $E$. 
Let $h$ be the class number of $K$. 
Choose any $\alpha \in \OO_K$ so that $\alpha \OO_K=\fq^{h}$. 
By Proposition~\ref{prop:lam12},
\[
\cN_\bs(\alpha) \equiv (\lambda(\sigma_\fq))^{12 h} \pmod{\fp_0}.
\]
However, if $\epsilon$ is unit, then $\epsilon \alpha \OO_K=\fq^h$
too. So
\[
\cN_\bs(\epsilon \alpha) \equiv (\lambda(\sigma_\fq))^{12 h} \pmod{\fp_0}.
\]
Taking ratios we have $\cN_\bs(\epsilon) \equiv 1 \pmod{\fp_0}$.
\end{proof}

Corollary~\ref{cor:unitbound} is only useful in bounding $p$
for a given signature $\bs$, if there is some unit $\epsilon$ of $K$
such that $\cN_\bs(\epsilon) \ne 1$. Of course, if $\bs$
is either of the constant signatures $(0)_{\tau \in G}$
or $(12)_{\tau \in G}$ then 
\[
\cN_\bs(\epsilon)=(\Norm{\epsilon})^{\text{$0$ or $12$}}=1.
\]
Given a non-constant signature $\bs \in S$, does there
exist a unit $\epsilon$ such that $\cN_\bs(\epsilon) \ne 1$?
It is easy construct examples where the answer is no.
The following lemma gives a positive answer when $K$
is totally real.
\begin{lem}\label{lem:dirichlet}
Let $K$ be totally real of degree $d \ge 2$.
Suppose $\bs \ne (0)_{\tau \in G}$, $\ne (12)_{\tau \in G}$. 
Then there exists a unit $\epsilon$ of $K$
such that $\cN_\bs(\epsilon) \ne 1$. 
\end{lem}
\begin{proof}
Let $\tau_1,\ldots,\tau_d$ be the elements of $G$. 
Fix an embedding $\sigma: K \hookrightarrow \R$,
and denote $\sigma_i=\sigma \circ \tau_i$. 
Rearranging the elements of $G$, we may suppose 
that 
\[
s_{\tau_1}=\cdots=s_{\tau_r}=12,
\qquad
s_{\tau_{r+1}}=\cdots=s_{\tau_d}=0
\]
where $1 \le r \le d-1$. Suppose that $\cN_\bs(\epsilon)=1$
for all $\epsilon \in U(K)$, where $U(K)$ is the unit group of $K$.
Then the image of $U(K)$ under the Dirichlet embedding
\[
U(K)/\{\pm 1\} \hookrightarrow \R^{d-1},\qquad
\epsilon \mapsto  
(\log{ \lvert \sigma_1(\epsilon)} \rvert,\ldots,
\log{ \lvert \sigma_{d-1}(\epsilon) }\rvert)
\]
is contained in the hyperplane $x_1+x_2+\cdots+x_r=0$. 
This contradicts the fact the image must be a lattice in $\R^{d-1}$ of rank $d-1$.
\end{proof}

\section{Proof of Theorem~\ref{thm:main}}
We now prove Theorem~\ref{thm:main}. Suppose $\overline{\rho}_{E,p}$
is reducible and let $\bs$ be the isogeny signature. 
Let $A_\bs$ be as in \eqref{eqn:As}. By Corollary~\ref{cor:unitbound},
$p \mid A_\bs$. If $\bs \ne (0)_{\tau \in G}$, $(12)_{\tau \in G}$
then $A_\bs \ne 0$ by Lemma~\ref{lem:dirichlet}.
Now, suppose $p \nmid B$, where $B$ is as in the statement of Theorem~\ref{thm:main}.
Then $\bs =(0)_{\tau \in G}$, $(12)_{\tau \in G}$.

Suppose first that $\bs=(0)_{\tau \in G}$. Then $\cN_\bs(\alpha)=1$
for all $\alpha$. 
Let $\fq \nmid p$ be a prime of good reduction for $E$.
It follows from Lemma~\ref{lem:good} that $\fp_0$ divides
the resultant of $P_\fq(X)$ and $X^{12r}-1$. As both polynomials
have coefficients in $\Z$, the resultant belongs to $\Z$, and so
is divisible by $p$. This completes the proof for $\bs=(0)_{\tau \in G}$.

Finally, we deal with the case $\bs=(12)_{\tau \in G}$.
Let $C \subset E[p]$ be the subgroup of order $p$ corresponding to $\lambda$. 
Replacing $E$ by the isogenous curve $E^\prime=E/C$
has the effect of swapping $\lambda$ and $\lambda^\prime$
in \eqref{eqn:matrix}.
As $\lambda \lambda^\prime=\chi_p$,
the isogeny signature for $E^\prime$ at $p$
is $(0)_{\tau \in G}$.  The theorem follows.

\section{Proof of Theorem~\ref{thm:2}}
Suppose $\overline{\rho}_{E,p}$ is reducible
with signature $\bs$.
As in the proof of Theorem~\ref{thm:main}, we may suppose
$\bs=(0)_{\tau \in G}$ or $\bs=(12)_{\tau \in G}$. 
Moreover, replacing $E$ by an isogenous elliptic curve
we may suppose that $\bs=(0)_{\tau \in G}$.
By definition of $\bs$, we have $\lambda^{12}$
is unramified at all $\fp \mid p$.
As is well-known (see for example \cite[Proposition 1.4 and Proposition 1.5]{DavidI}),
$\lambda^{12}$ is unramified at the finite places outside $p$;
$\lambda^{12}$ is clearly unramified at the infinite places because of the
even exponent $12$. 
Thus ${\lambda}^{12}$ 
is everywhere unramified. 
Thus $\lambda$ has order dividing $12 \cdot h$, where $h$
is the class number of $K$. 
Let $L/K$ be the extension cut out by $\lambda$; 
this has degree dividing $12 \cdot h$.
Then $E/L$ has a point of order $p$. Applying Merel's bounds \cite{Merel},
we conclude that
\[
p < (1+3^{[L:\Q]/2})^2 \le (1+3^{6 d h})^2.
\]

\section{An Example: Frey Curves Attached to Fermat Equations of Signature $(13,13,p)$}

In \cite{DF2}, Dieulefait and Freitas, 
used the modular method to attack certain 
Fermat-type equations of the form $x^{13} + y^{13} = Cz^p$, for infinitely many values of $C$. 
They attach Frey curves (independent of $C$) over $\Q(\sqrt{13})$
to primitive solutions 
of these equations, and  prove irreducibility of the mod $p$ representations attached to these Frey curves,
for $p>7$ and $p \ne 13$, $37$ under the assumption that the isogeny signatures are $(0,0)$ or $(12,12)$. 
Here we improve on the argument by dealing with the isogeny signature $(0,12)$, $(12,0)$
and also by dealing with $p=37$.
More precisely, we prove the following.
\begin{thm}
\label{thm:13p}
 Let $d=3$, $5$, $7$ or $11$
and let $\gamma$ be an integer divisible only by primes
$\ell \not \equiv 1 \pmod{13}$. Let $p$ be a prime satisfying
$p\ge 17$ or $p=11$.
Let $(a,b,c) \in \Z^3$ satisfy 
\[
a^{13}+b^{13}=d \gamma c^p, \qquad \gcd(a,b)=1, \qquad abc \ne 0,\pm 1.
\]
Write $K = \Q(\sqrt{13})$; this has class number $1$.
Let $E = E_{(a,b)}/K$ be the Frey curve 
defined in \cite{DF2}. 
Then, the Galois representation $\overline{\rho}_{E,p}$ is irreducible.
\label{correction2}
\end{thm}
\begin{proof} 
Suppose $\overline{\rho}_{E,p}$ is reducible.
For a quadratic field such as $K$, the set $S$ of 
possible isogeny signatures $(s_\tau)_{\tau \in G}$ is 
\[
S=\{(12,12), (12,0), (0,12), (0,0)\}.
\] 
Note that $(13) = (\sqrt{13})^2$ is the only prime ramifying in $K$. 
In \cite{DF2} it is shown that the curves $E$ have 
additive reduction only at $2$ and $\sqrt{13}$. 
Moreover, $E$ has
good reduction at all primes $\fq \nmid 26$ 
above rational primes $q \not \equiv 1 \pmod{13}$.
Furthermore, the trace $a_\fq(E_{(a,b)})$ depends
only on the values of $a$, $b$ modulo $q$. 
By the assumption $\gcd(a,b)=1$, we have $(a,b) \not \equiv 0 \pmod{q}$.

The fundamental unit of $K$ is $\epsilon=(3+\sqrt{13})/2$. Then
\[
\Norm(\epsilon^{12}-1)=-2^{6} \cdot 3^4 \cdot 5^2 \cdot 13.
\]
As $p\ge 17$ or $p=11$, it follows from Corollary~\ref{cor:unitbound}
that the isogeny signature of $E$ at $p$
is either $(0,0)$ or $(12,12)$. As in the proof of Theorem~\ref{thm:main},
we may suppose that the isogeny signature is $(0,0)$.

Let $q$ be a rational prime $\not \equiv 1 \pmod{13}$ that splits
as $(q)= \fq_1 \cdot \fq_2$ in $K$. By the above, $\fq_1$, $\fq_2$
must be primes of good reduction. The trace $a_{\fq_i}(E_{(a,b)})$
depends only on the values of $a$, $b$ modulo $q$.
For each non-zero pair $(a,b) \pmod{q}$, let 
\begin{equation}
P_{\fq_1}^{(a,b)}(X)=X^2 - a_{\fq_1}(E_{(a,b)})X + q \qquad 
P_{\fq_2}^{(a,b)}(X)=x^2 - a_{\fq_2}(E_{(a,b)})x + q, 
\label{polys}
\end{equation}
be the characteristic polynomials of Frobenius
at $\fq_1$, $\fq_2$. Let
\[
R_q^{a,b}=\gcd(\, \Res(P_{\fq_1}^{a,b}(X),X^{12}-1)\, ,\, \Res(P_{\fq_2}^{a,b}(X),X^{12}-1)\, ).
\]
Let
\[
R_q=\lcm \{ R_q^{a,b} \; : 0 \le a,b \le q-1, \qquad (a,b) \ne (0,0)  \}.
\]
By the proof of Theorem~\ref{thm:main},
we have that $p$ divides 
$R_q$. Using a short {\tt SAGE} script we computed the values of $R_q$ for 
$q=3$, $17$. We have
\[
R_3=
2^6 \cdot 3^2 \cdot 5^2 \cdot 37,
\qquad
R_{17}=
2^8 \cdot 3^4 \cdot 5^2 \cdot 7^2 \cdot 13^2 \cdot 19 \cdot 23 \cdot 53 \cdot 
97 \cdot 281 \cdot 21481 \cdot 22777.
\]
As $p \ge 17$ or $p=11$ we see that $\overline{\rho}_{E,p}$ is irreducible.
\end{proof}

We will now use the improved irreducibility result 
(Theorem~\ref{thm:13p})
to correctly restate Theorem~1.3 in \cite{DF2}. 
Furthermore, we will also add an argument using the primes above $17$ that actually allows us to improve it. More precisely, we will prove.

\begin{thm} Let $d=3$, $5$, $7$ or $11$
and let $\gamma$ be an integer divisible only by primes
$\ell \not \equiv 1 \pmod{13}$. Let also $\mathcal{L} := \{2, 3, 5, 7, 11, 13, 19, 23, 29, 71\}.$ 

If $p$ is a prime not belonging to $\mathcal{L}$, then: 
\begin{itemize}
\item[(I)] The equation $x^{13} + y^{13} = d\gamma z^p$ has no solutions $(a,b,c)$ such that 
\[
\gcd(a,b)=1, \qquad abc \ne 0,\pm 1 \qquad \text{ and } \qquad 13 \nmid c. 
\]
\item[(II)] The equation $x^{26} + y^{26} = 10\gamma z^p$ has no solutions $(a,b,c)$ such that
\[
\gcd(a,b)=1, \qquad \text{and} \qquad abc \ne 0,\pm 1. 
\]
\end{itemize}
\end{thm}
\begin{proof} Suppose there is a solution $(a,b,c)$, satisfying $\gcd(a,b)=1$, to the equation in part (I) of the theorem for $p \geq 17$ or $p=11$. Let $E = E_{(a,b)}$ be the Frey curves attached to it in \cite{DF2}. As explained in \cite{DF2}, but now using Theorem~\ref{thm:13p} above instead of Theorem 4.1 in \textit{loc. cit.}, we obtain an isomorphism 
\begin{equation}
\bar{\rho}_{E,p} \sim \bar{\rho}_{f,\fp}, 
\label{correction}
\end{equation}
where $\fp \mid p$ and $f \in S_2 (2^{i}w^{2})$ for $i=3,4$. In \textit{loc. cit} the newforms are divided into the sets
\begin{itemize}
 \item [S1:] The newforms in $S_2 (2^{i}w^{2})$ for $i=3,4$ such that $\Q_f = \Q$,
 \item [S2:] The newforms in the same levels with $\Q_f$ strictly containing $\Q$.
\end{itemize}
We eliminate the newforms in S1 with the same argument as in \cite{DF2}. Suppose now that isomorphism \eqref{correction} holds with a form in S2. Also in \cite{DF2}, using the primes dividing 3, a contradiction is obtained if we assume that 
\[ 
p \not\in \mathcal{P} = \{2, 3, 5, 7, 11, 13, 19, 23, 29, 71, 191, 251, 439, 1511, 13649\}.
\]
Going through analogous computations, 
using the two primes dividing $17$, 
gives a contradiction if $p$ does not belong to 
\begin{multline*}
\mathcal{P}^\prime = \{ 
2, 3, 5, 7, 11, 13, 17, 19, 23, 29, 37, 41, 43, 47, 59, 71, 73, 79, 83, 89, 
109,\\ 113, 157, 167, 197, 227, 229, 239, 281, 359, 431, 461, 541, 1429, 5237, 
253273,\\ 271499, 609979, 6125701, 93797731, 530547937, 733958569, 6075773983,\\
11740264873 
\}.
\end{multline*}
Thus, we have a contradiction as long as $p$ is not in the intersection
\[
\mathcal{P} \cap \mathcal{P}^\prime = \{ 2, 3, 5, 7, 11, 13, 19, 23, 29, 71\}. 
\]
Thus part (I) of the theorem follows. Part (II) follows exactly as in \cite{DF2}
\end{proof}

\end{document}